\documentclass[12pt]{amsart}
\usepackage{amssymb, amsmath, amsfonts, amsthm, esint}
\usepackage[abbrev]{amsrefs}
\usepackage[margin=0.9 in]{geometry}
\usepackage[colorlinks=true, linkcolor=blue]{hyperref}

\def \dd {\partial}

\def \eps {\varepsilon}
 
\DeclareMathOperator{\Ric}{Ric}

\DeclareMathOperator{\Hess}{Hess}
\DeclareMathOperator{\diam}{diam}
\DeclareMathOperator{\vol}{vol}

\DeclareMathOperator{\diver}{div}

\DeclareMathOperator{\area}{area}

\newcommand{\be}{\begin{equation}}
\newcommand{\ee}{\end{equation}}
\newcommand{\ba}{\begin{eqnarray}}
\newcommand{\ea}{\end{eqnarray}}
\newcommand{\ban}{\begin{eqnarray*}}
\newcommand{\ean}{\end{eqnarray*}}

\title{First Eigenvalue of the $p$-Laplacian under integral curvature condition}
\author{Shoo Seto}
\address{Department of Mathematics\\
         University of California\\
         Santa Barbara, CA 93106}
\email{\href{mailto:shoseto@ucsb.edu}{shoseto@ucsb.edu}}
\thanks{Partially support by Simons Travel Grant}
\author{Guofang Wei}
\address{Department of Mathematics\\
         University of California\\
         Santa Barbara, CA 93106}
\email{\href{mailto:wei@math.ucsb.edu}{wei@math.ucsb.edu}}
\thanks{Partially supported by NSF DMS 1506393}
\date{}
\keywords{$p$-Laplacian, first eigenvalue, integral Ricci curvature}

\theoremstyle{definition}

\newtheorem{lemma}{Lemma}[section]

\newtheorem{proposition}{Proposition}[section]
\newtheorem{theorem}{Theorem}[section]

\numberwithin{equation}{section}

\begin{document}
\maketitle

\begin{abstract}
We give various estimates of the first eigenvalue of the $p$-Laplace operator on closed Riemannian manifold with integral curvature conditions.
\end{abstract}

\section{Introduction}
On a compact Riemannian manifold $(M^n,g)$,  for $1<p< \infty$, the  $p$-Laplacian is defined by
\begin{equation}
\Delta_p(f) := \diver(|\nabla f|^{p-2}\nabla f). \label{p-Lap-def}
\end{equation}
It is a second order quasilinear elliptic operator and when $p=2$ it is the usual Laplacian. The  $p$-Laplacian has applications in many
different contexts from game theory to mechanics and image processing.   Corresponding to the $p$-Laplacian, we have the eigenvalue equation 
\begin{equation}\label{eigen-eq}
\begin{cases}
\Delta_p(f) = -\lambda |f|^{p-2}f & \text{ on } M \\
\nabla_{\nu}f \equiv 0 \text{ (Neumann) or }
f \equiv 0 \text{ (Dirichlet)} & \text{ on }\dd M
\end{cases} 
\end{equation}
where $\nu$ is the outward normal on $\dd M$.  The first nontrivial Neumann eigenvalue for $M$ is given by
\begin{equation}
\mu_{1,p}(M) = \inf\left\{ \frac{\int_M |\nabla f|^p}{\int_M |f|^p} \ | \ f \in W^{1,p}(M)\backslash \{0\}, \int_M|f|^{p-2}f = 0 \right\}
\end{equation}
and the first Dirichlet eigenvalue of $M$ is given by
\begin{equation}  \label{lambda}
\lambda_{1,p}(M) = \inf\left\{ \frac{\int_M |\nabla f|^p}{\int_M |f|^p} \ | \ f \in W^{1,p}_c(M)\backslash \{0\} \right\}.
\end{equation}

Though the regularity theory of the p-Laplacian is very different from the usual Laplacian, 
many of the estimates for the first eigenvalue of the Laplacian (when $p=2$) can be generalized to general $p$.  Matei  \cite{matei} generalized  Cheng's first Dirichlet eigenvalue comparison of balls \cite{cheng} to the $p$-Laplacian. 
 For closed manifolds with Ricci curvature bounded below by $(n-1)K$,  Matei for $K>0$ \cite{matei},
 Valtora for $K=0$ \cite{valtorta} and Naber-Valtora for general $K\in \mathbb{R}$ \cite{naber-valtorta} give a sharp lower bound for the first nontrivial eigenvalue. Andrews-Clutterbuck \cite{andrews-survey},\cite{andrews-clutterbuck} also gave a proof using modulus of continuity argument. L.F. Wang \cite{wang} considered the case when the Bakry-Emery curvature has a positive lower bound for weighted $p$-Laplacians. Recently Y.-Z. Wang and H.-Q. Li \cite{wang-li} extended the estimates to smooth metric measure space and  Cavalletti-Mondino \cite{cavalletti-mondino} to general  metric measure space.  For a general reference on the $p$-Laplace equation, see \cite{lindqvist}. See also \cite{Zhang2007} and references in the paper for  related lower bound estimates.

In this paper, we extend the first eigenvalue estimates for $p$-Laplacian given in \cite{matei} to the integral Ricci curvature setting.

For each $x\in M^n$ let $\rho\left( x\right) $ denote the smallest
eigenvalue for the Ricci tensor $\mathrm{Ric}:T_{x}M\rightarrow
T_{x}M,$ and $\Ric_-^K(x) = \left( (n-1)K - \rho
(x)\right)_+ = \max \left\{ 0, (n-1)K - \rho (x) \right\}$, the amount of Ricci curvature lying below $(n-1)K$. Let
\be \| \Ric_-^K \|_{q, R} = \sup_{x\in M} \left( \int_{B\left( x,R\right) } (\mathrm{Ric}_-^K)^{q}\,
d vol\right)^{\frac 1q}. \ee
Then  $\| \mathrm{Ric}_-^K \|_{q,R} $ measures the amount of Ricci
curvature lying below a given bound, in this case,  $(n-1)K$, in the $L^q$ sense. Clearly $\|
\mathrm{Ric}_-^K \|_{q, R} = 0$ iff $\Ric_M \ge (n-1)K$.  Denote the limit as $R\to \infty$ by $\|\Ric_-^K\|_q$, which is a global curvature invariant.  The Laplace and volume comparison, the basic tools for manifolds with pointwise Ricci curvature lower bound,  have been extended to integral Ricci curvature bound \cite{petersen-wei97}, see Theorem~\ref{com-int}. 

We denote $\|f\|_{q,\Omega}^*$ the normalized $q$-norm on the domain $\Omega$. Namely
\[
\|f\|_{q, \Omega}^* = \left(\frac{1}{\vol(\Omega)}\int_\Omega |f|^q \right)^{\frac 1q}.  \]

%We also introduce the normalized integral Ricci curvature by
%\begin{equation*}
%\bar k(q,K) = \sup_{x\in M}\frac{1}{\vol(M)}\int_M (\Ric_-^K)^q
%\end{equation*}

Under the assumption that the integral Ricci curvature is controlled ($\|\Ric_-^K\|_q^*$ is small), we give the following first eigenvalue estimates:

\begin{theorem}[Cheng-type estimate]\label{Cheng}
Let $(M^n,g)$ be a complete Riemannian manifold. For any $x_0\in M$, $K\in \mathbb{R}$,  $r>0$,  $p> 1$, $q >\frac{n}{2}$, denote  $\bar{q} = \max\{q, \frac p2 \}$,  there exists an $\eps = \eps(n,p,\bar{q},K,r)$ such that  if $\partial B(x_0, r) \not= \emptyset$ and $\|\Ric_-^K\|_{\bar{q}, B(x_0, r)}^* <\eps$,   then
\begin{equation*}
\lambda_{1,p}(B(x_0,r))\leq \bar{\lambda}_{1,p}(B_K(r))+ C(n,p,\bar{q},K,r ) \left( \|\Ric_-^K\|_{\bar{q}, B(x_0, r)}^*\right)^{\frac 12},
\end{equation*}
where $\mathbb{M}^n_K$ is the complete simply connected space of constant curvature $K$, $B_K(r) \subset \mathbb{M}^n_K$ is the ball of radius $r$ and $\bar\lambda_{1,p}$ is the first Dirichlet eigenvalue of the $p$-Laplacian in the model space $\mathbb{M}^n_K$.
\end{theorem}
This generalizes the Dirichlet $p$-Laplacian first eigenvalue comparison in \cite{matei}. When $p = 2$, this is proved in \cite{petersen-sprouse}.

\begin{theorem}[Lichnerowicz-type estimate]  \label{N-eigen-int}
	Let $(M^n,g)$ be a complete Riemannian manifold.  For $q > \frac{n}{2}$, $p \geq 2$ and $K>0$, there exists $\eps = \eps(n,p,q,K)$ such that if $ \|\Ric_-^K\|^*_q \le \eps$, then 
	\begin{equation}
	\mu_{1,p}^{\frac{2}{p}} \geq \frac{\sqrt{n}(p-2)+n}{(p-1)(\sqrt{n}(p-2)+n-1)}
	\left[(n-1)K-2 \|\Ric_-^K\|^*_q\right].  \label{N-eigen}
	\end{equation}
 In particular, when $\Ric \ge (n-1)K$, we have 
	\begin{equation}
	\mu_{1,p}^{\frac{2}{p}} \geq \frac{\sqrt{n}(p-2)+n}{\sqrt{n}(p-2)+n-1} \cdot \frac{(n-1)K}{p-1} \ge \frac{(n-1)K}{p-1}.  \label{explicit-est} 
	\end{equation}
\end{theorem}

Under these assumption, Aubry's diameter estimate implies that $M$ is closed \cite[]{aubry}. That paper also has the proof for $p=2$. 

The explicit estimate \eqref{explicit-est} improves the estimate in \cite[Theorem 3.2]{matei}, where it is shown that $ (\mu_{1,p})^{\frac{2}{p}} \geq \frac{(n-1)K}{p-1}$. When $p=2$, the estimate \eqref{explicit-est} recovers the Lichnerowicz estimate that $\mu_{1,2} \ge nK$.  The explicit  estimate \eqref{N-eigen} is optimal when $p=2$, but not optimal when $p>2$. For optimal estimate we have the following 
Lichnerowicz-Obata-type estimate. 
\begin{theorem}[Lichnerowicz-Obata-type estimate]  \label{optimal-est}
	Let $M^n$ be a complete Riemannian manifold.  Then for any $\alpha >1$, $K>0$, $q> \frac n2$ and any $p>1$, there is an $\eps=\eps(n,p,q,\alpha,K) >0$ such that if $\|\Ric_-^K\|_q^* \leq \eps$, then 
	\begin{equation*}
	\alpha\mu_{1,p}(M) \geq \mu_{1,p}(\mathbb M_K^n).
	\end{equation*}
\end{theorem}
When $\|\Ric_-^K\|_q^* = 0$, we can take $\alpha =1$ and this gives Theorem~3.1 in \cite{matei}. 

This result is obtained from the following Faber-Krahn type estimate. Recall the classical Faber-Krahn inequality asserts that  in $\mathbb R^n$ balls (uniquely) minimize the first eigenvalue of the Dirichlet-Laplacian among sets with given volume.
\begin{theorem}[Faber-Krahn-type estimate]\label{Faber-Krahn}
Under the same set up as in Theorem~\ref{optimal-est},  let $\Omega \subset M$ be a domain and $B_K \subset \mathbb M_K^n$ be a geodesic ball in the model space such that
\begin{equation*}
\frac{\vol(\Omega)}{\vol(M)} = \frac{\vol(B_K)}{\vol(M^n_K)}.
\end{equation*}
Then 
\begin{equation*}
\alpha^p\lambda_{1,p}(\Omega) \geq \lambda_{1,p}(B_K).
\end{equation*}
\end{theorem}
 Again when $\|\Ric_-^K\|_q^* = 0$, we can take $\alpha =1$ and this gives Theorem~2.1 in \cite{matei}.

%\color{red}When we have a point-wise Ricci lower bound one can show some rigidity, however for the integral curvature case we do not.  The key difference between the $p=2$ case and general case is that we are not able to directly use maximum principle to obtain the eigenvalue estimates. Likewise, for integral curvature, the lack of a pointwise bound leads us to use techniques besides the maximum principle. Hence combining the two situation seems to be a natural generalization.\color{black}

To prove these results, since we do not have pointwise Ricci curvature lower bound, one key is to control the error terms.
 
We now give a quick overview of the paper.  In \S 2 we prove the Cheng-type upper bound using the first eigenfunction of $\Delta_p$ for the model case as a test function in the $L^p$-Rayleigh quotient and using the Laplacian comparison and volume doubling for integral curvature (Theorem~2.1) to control the error.  In \S 3, we prove the Lichnerowicz-type lower bound by using the $p$-Bochner formula and the Sobolev inequality.  In \S 4, to prove a Faber-Krahn-type lower bound, a necessary tool we need is an integral curvature version of the Gromov-Levy isoperimetric inequality, which we first show.  The proof of the eigenvalue estimate then follows from an argument using the co-area formula. 

\section{Proof of Theorem \ref{Cheng}}
First we recall the Laplace and volume comparison for integral Ricci curvature proved by the second author joint with Petersen \cite{petersen-wei97, petersen-wei00}.

 Let $M^n$ be a complete Riemannian manifold of dimension $n$. Given $x_0 \in M$, let $r(x) = d(x_0, x)$ be the distance function and $\psi (x) = \left( \Delta r - \bar{\Delta}^K r \right) _{+}$, where $\bar{\Delta}^K$ is the Laplacian on the model space $\mathbb{M}^n_K$. 
 The classical Laplace comparison states that if $\Ric_M \ge (n-1)K$, then $\Delta r \le \bar{\Delta}_K r$, i.e., if $\Ric_-^K \equiv 0$, then $\psi  \equiv 0$. In \cite{petersen-wei97} this is generalized to integral Ricci lower bound.
 
 \begin{theorem}[Laplace and Volume Comparison \cite{petersen-wei97, petersen-wei00}]  \label{com-int}
 	Let $M^n$ be a complete Riemannian manifold of dimension $n$. If $q>\frac{n}{2}$, then
 	\be
 	\label{Laplacian-com-average-norm}
 	\| \psi \|^*_{2q. B(x,r)}  \leq C(n,q) \left( \|
 	\mathrm{Ric}_-^K \|^*_{q, B(x, r)} \right)^{\frac 12}. \ee
 	There exists $\varepsilon=\varepsilon(n,q,K,r)>0$ such that, if $\|\Ric_-^K\|_{q, B(x, r)}^* \le\varepsilon,$
 	then
 	\begin{equation}\label{volume doubling}
 	\frac{\vol(B(x,r))}{\vol(B(x,r_0)}\le 2 \frac{\vol B_K(r)}{\vol B_K (r_0)},\, \ \ \forall r_0 \le r.
 	\end{equation}
 \end{theorem}

For $p$-Laplacian of radial function, we have the following comparison.
\begin{proposition}[$p$-Laplace comparison]
	If $f$ is a radial function such that $f' \le 0$, then
	\be
	\Delta_p f \ge 	\bar{\Delta}^K_p f + f'|f'|^{p-2} \psi.  \label{p-com}
	\ee
\end{proposition}	
\begin{proof}
From the definition of the $p$-Laplacian \eqref{p-Lap-def},
\begin{align}
\begin{split}
\Delta_p f = \diver(|\nabla f|^{p-2}\nabla f) &= \langle \nabla|\nabla f|^{p-2},\nabla f\rangle + |\nabla f|^{p-2}\Delta f\\
&= (p-2)|\nabla f|^{p-4}\Hess f(\nabla f,\nabla f) + |\nabla f|^{p-2}\Delta f. \label{p-Lap}
\end{split}
\end{align}
Hence when $f = f(r)$ is a radial function
\begin{align}
\Delta_p f &= (p-2)|f'|^{p-2}f''+|f'|^{p-2}(f''+ \Delta r f')  \label{Lp-radial} \\
&= (p-1)|f'|^{p-2}f'' + \Delta r f'|f'|^{p-2} \nonumber \\
&= (p-1)|f'|^{p-2}f'' + \bar{\Delta}^K r f'|f'|^{p-2} +\left( \Delta r- \bar{\Delta}^K r\right) f'|f'|^{p-2} \nonumber \\
& = \bar{\Delta}^K_p r + \left( \Delta r- \bar{\Delta}^K r\right) f'|f'|^{p-2}. \nonumber
\end{align}
When $f' \le 0$, $\left( \Delta r- \bar{\Delta}^K r\right) f'|f'|^{p-2} \ge \psi f'|f'|^{p-2}$, which gives the estimate. 
\end{proof}
 
 Let $\bar{f} >0$ be the first eigenfunction for the Dirichlet problem for $\Delta_p$ in $B_K(r) \subset \mathbb{M}^n_K$. By \cite{delpino-manasevich} $\bar{f}$ is radial. Below we show that $\bar{f}$ is a decreasing function of the radius.  For $p\geq 2$, this was shown in \cite{matei}. Our proof is much shorter. 
\begin{lemma}
For $t \in (0, r)$ and $p>1$,   $\bar{f}'(t) \leq 0$.
\end{lemma} 

\begin{proof}
Write the volume element of $\mathbb{M}^n_K$ in geodesic polar coordinate $dvol = \mathcal{A}(t) dt d\theta_{S^{n-1}}$. As the first eigenfunction $\bar{f}$ is radial, by (\ref{Lp-radial}) it satisfies the ODE
\begin{equation}
(\mathcal{A}|f'|^{p-2}f')' = -\lambda_1\mathcal{A}|f|^{p-2}f.
\end{equation}
 As $\mathcal A (0) =0$ and $p>1$, integrating both sides from $0$ to $t$ we get
\begin{equation*}
\mathcal{A}|f'|^{p-2}f'(t) = -\lambda_1\int_0^t|f|^{p-2}f\mathcal{A} \leq 0.
\end{equation*}
\end{proof}

Now we are ready to prove Theorem~\ref{Cheng}.  
\begin{proof}
Let $\bar{f}$ be a first eigenfunction for the Dirichlet problem for $\Delta_p$ in $B_K(r) \subset \mathbb M^n_K$ with $\bar{f} (0) =1$.  Hence $0 \le \bar{f} \le 1$. Let $r= r(x) = d(x_0,x)$ be the distance function on $M$ centered at the point $x_0$.  Then $\bar{f}(r) \in W_0^{1,p}(B(x_0,r))$.
Denote $Q = \frac{\int_{B}|\nabla \bar{f}|^p}{\int_B |\bar{f}|^p}$, where $B:=B(x_0,r)$. 
 By (\ref{lambda}) we have,
 \be
 \lambda_{1,p}(B(x_0,r))  \le Q.
 \ee
Using integration by part, $\bar{f}' \leq 0$, $0 \le \bar{f} \le 1$ and the p-Laplacian comparison (\ref{p-com}), we have 
\begin{align*}
Q &= - \frac{\int_B \Delta_p \bar{f} \cdot \bar{f}}
{\int_B |\bar{f}|^p } \\
& \le - \frac{\int_B \bar{\Delta}^K_p \bar{f} \cdot \bar{f} + \psi \bar{f}'|\bar{f}'|^{p-2} \bar{f} }
{\int_B |\bar{f}|^p }\\
& = \bar{\lambda}_{1,p}(B_K(r)) - \frac{\int_B  \psi \bar{f}'|\bar{f}'|^{p-2} \bar{f} }
{\int_B |\bar{f}|^p } \\
& \le \bar{\lambda}_{1,p}(B_K(r)) + \frac{\int_B  \psi |\bar{f}'|^{p-1} }
{\int_B |\bar{f}|^p }.
\end{align*}
By H\"older inequality
\[\int_B  \psi |\bar{f}'|^{p-1} \le \left(\int_B  \psi^{p}\right)^{\frac{1}{p}} \left(\int_B |\bar{f}'|^{p}\right)^{1-\frac{1}{p}}.
\]
Let $r_0 = r_0(n,K,r) \in (0, r)$ such that $\bar{f}(r_0) =\frac 12$.   Then $\bar{f} \ge \frac 12$ on  $[0,r_0]$, and 
\[\int_B |\bar{f}|^p \ge \left( \int_B |\bar{f}|^p \right)^{1-\frac 1p} \cdot \left( \int_{B(x_0, r_0)} |\bar{f}|^p \right)^{\frac 1p} \ge \left( \int_B |\bar{f}|^p \right)^{1-\frac 1p} \cdot \left( \vol B(x_0, r_0) 2^{-p} \right)^{\frac 1p}.
\]
Hence the error term 
\ban 
\frac{\int_B  \psi |\bar{f}'|^{p-1} }
{\int_B |\bar{f}|^p } & \le &  2 \left(\frac{\int_B \psi^p}{\vol B(x_0, r_0)} \right)^{\frac 1p}  \cdot \left(\frac{ \int_B |\bar{f'}|^p}{ \int_B |\bar{f}|^p} \right)^{1-\frac 1p}   \\
& = & 2\, Q^{1-\frac 1p} \left(\fint_B \psi^p\right)^{\frac 1p} \left(\frac{\vol B(x_0,r)}{\vol B(x_0, r_0)} \right)^{\frac 1p} \\
& \le & 2\, Q^{1-\frac 1p} \| \psi\|^*_{2\bar{q}, B(x_0,r)} \left(\frac{\vol B(x_0,r)}{\vol B(x_0, r_0)} \right)^{\frac 1p}. 
\ean
Choose $\eps \le \eps_0$ in Theorem~\ref{com-int},  using (\ref{Laplacian-com-average-norm}) and (\ref{volume doubling}), and combining above, we have
\be \label{initialineq}
Q \le \bar{\lambda}_{1,p}(B_K(r)) + C(n,p,\bar{q},K,r ) \left( \|\Ric_-^K\|_{\bar{q}, B(x_0, r)}^*\right)^{\frac 12} Q^{1-\frac 1p}.
\ee

Applying Young's inequality to the last term, we have
\begin{equation*}
Q \leq \bar{\lambda}_{1,p}(B_K(r)) + \frac{1}{p}C(n,p,\bar{q},K,r ) \left( \|\Ric_-^K\|_{\bar{q}, B(x_0, r)}^*\right)^{\frac p2} + \frac{p-1}{p}Q.
\end{equation*}
Moving $Q$ to the left hand side, we obtain
\begin{equation*}
Q \leq p\bar{\lambda}_{1,p}(B_K(r)) +C(n,p,\bar{q},K,r ) \left( \|\Ric_-^K\|_{\bar{q}, B(x_0, r)}^*\right)^{\frac p2}.
\end{equation*}
Applying this to \eqref{initialineq} so that the $\displaystyle Q^{1-\frac{1}{p}}$ can be bounded in terms of the fixed quantities, we obtain
\begin{equation*}
Q \leq \bar{\lambda}_{1,p}(B_K(r)) +C(n,p,\bar{q},K,r ) \left( \|\Ric_-^K\|_{\bar{q}, B(x_0, r)}^*\right)^{\frac 12}.
\end{equation*}
%\color{red}
%write as $(1+\delta)\bar{\lambda}$ instead since the final $C(n,p,\bar q,K,r)$ contains $\bar{\lambda}$?
%\color{black}
\end{proof}

\section{Proof of Theorem~\ref{N-eigen-int}}
To prove Theorem~\ref{N-eigen-int}, we need the following Bochner formula for $p$ power. 
\begin{lemma}[$p$-Bochner]
\begin{align}
\begin{split}
\frac{1}{p}&\Delta (|\nabla f|^p)\\
&=(p-2)|\nabla f|^{p-2} |\nabla  | \nabla f | |^2 +\frac{1}{2}|\nabla f|^{p-2}\left\{  |\Hess f|^2 + \langle \nabla f,  \nabla \Delta f \rangle + \Ric(\nabla f,\nabla f) \right\}.   \label{p-bochner}
\end{split}
\end{align}
\end{lemma}
 One can find this implicitly in the literature, see e.g. \cite[]{colding, naber-valtorta, song-wei-wu}. The proof is very simple, for completeness, we present it here. 
\begin{proof} One computes
\begin{align} 
\frac{1}{p}\Delta (|\nabla f|^p)  =   \frac{1}{p}\Delta (|\nabla f|^2)^{\frac p2}  
 =  (p-2)|\nabla f|^{p-2} |\nabla  | \nabla f | |^2 +\frac{1}{2}|\nabla f|^{p-2}\Delta(|\nabla f|^2).  \label{Lap-p}
\end{align}
Recall the Bochner formula
\begin{equation*}
\frac 12 \Delta(|\nabla f|^2)  = |\Hess f|^2 + \langle \nabla f,  \nabla \Delta f \rangle + \Ric(\nabla f,\nabla f),   \label{bochner}
\end{equation*}
Plugging this into (\ref{Lap-p}) gives (\ref{p-bochner}). 
\end{proof} 

We also need the following Sobolev inequality, which follows from Gallot's isoperimetric constant estimate for integral curvature \cite[]{gallot} and Aubry's diameter estimate \cite[]{aubry}.
\begin{proposition}
Given $q > \frac n2$ and $K>0$, there exists $\eps= \eps(n,q,K)$ such that if  $M^n$ is a complete manifold with $\|\Ric_-^K\|^*_q \le \eps$, then there is a constant $C_s(n,q,K)$ such that 
\begin{equation}
	\left(\fint_M f^{\frac{2q}{q-1}}\right)^{\frac{q-1}{q}} \leq C_s(n,q, K) \fint_M |\nabla f|^2 + 2 \fint_M f^2 \label{Sobolev}
	\end{equation}
	for all functions $f \in W^{1,2}$.
\end{proposition}

Now we are ready to prove Theorem~\ref{N-eigen-int}.
\begin{proof}[Proof of Theorem \ref{N-eigen-int}] When $p=2$ the result is proved in \cite[]{aubry}. In the rest we assume $p>2$.
	
By Aubry's diameter estimate \cite[]{aubry}, $M$ is closed. Integrating (\ref{p-bochner}) on $M$ we have 
\begin{equation}
 0 = \fint_M |\nabla f|^{p-2}\left\{(p-2)|\nabla|\nabla f||^2+  |\Hess f|^2 + \langle \nabla f,  \nabla \Delta f \rangle + \Ric(\nabla f,\nabla f) \right\}.   \label{Boch-int} \end{equation}
For the Hessian term, using the Cauchy-Schwarz inequalities
 \begin{equation*}
 |\Hess(\nabla f,\nabla f)|^2 \leq |\nabla f|^4|\Hess f|^2, 
 \end{equation*}
 \begin{equation*}
 |\Delta f|^2 \leq n |\Hess f|^2
 \end{equation*}
 and the formula for $p$-Laplacian (\ref{p-Lap}), we have 
\begin{align*}
\fint_M|\nabla f|^{p-2}|\Hess f|^2 &\geq \fint_M|\nabla f|^{p-2}\frac{(\Delta f)^2}{n} \\
&=\frac{1}{n}\fint_M\Delta f\Delta_p f- \frac{p-2}{n}\fint_M\Delta f |\nabla f|^{p-4}\Hess f(\nabla f,\nabla f) \\
&  \geq \frac{1}{n}\fint_M\Delta f\Delta_p f- \frac{p-2}{n}\fint_M |\Delta f| |\nabla f|^{p-2}|\Hess f| \\
&\geq \frac{1}{n}\fint_M\Delta f\Delta_p f- \frac{p-2}{\sqrt{n}}\fint_M |\nabla f|^{p-2}|\Hess f|^2.
\end{align*}
Hence
\begin{equation}
\fint_M |\nabla f|^{p-2}|\Hess f|^2 \geq \frac{1}{n+ \sqrt{n}(p-2)}\fint \Delta f\Delta_p f.
\end{equation}
For the third term, 
\[ \fint_M|\nabla f|^{p-2}\langle\nabla f,\nabla (\Delta f)\rangle = -\fint_M\Delta_pf\Delta f.
\]
For the curvature term,
\begin{align*}
\fint_M|\nabla f|^{p-2}\Ric(\nabla f,\nabla f) &\geq (n-1)K\fint_M |\nabla f|^p - \fint_M|\Ric_-^K||\nabla f|^p \\
&\geq (n-1)K\fint_M |\nabla f|^p - \|\Ric_-^K\|_q^*\left(\fint_M |\nabla f|^{\frac{pq}{q-1}}\right)^{\frac{q-1}{q}}.
\end{align*}
Apply the Sobolev inequality (\ref{Sobolev}) to the function $|\nabla f|^{\frac p2}$ gives 
\begin{align*}
\left(\fint_M \left(|\nabla f|^{\frac{p}{2}}\right)^{\frac{2q}{q-1}}\right)^{\frac{q-1}{q}} &\leq C_s \fint_M|\nabla|\nabla f|^{\frac{p}{2}}|^2 + 2 \fint_M |\nabla f|^p \\
&=C_s\frac{p^2}{4}\fint_M|\nabla f|^{p-2}|\nabla|\nabla f||^2 + 2 \fint_M|\nabla f|^p.
\end{align*}
Plug these to (\ref{Boch-int}), we have 
\begin{align*}
0 &\geq - \frac{n-1+\sqrt{n}(p-2)}{n+ \sqrt{n}(p-2)}  \fint_M \Delta_pf \Delta f + ((n-1)K-2\|\Ric_-^K\|^*_q)\fint_M|\nabla f|^p \\
&\hspace{0.2 in}+\left((p-2)-C_s\|\Ric_-^K\|_q^*\frac{p^2}{4}\right)\fint_M|\nabla f|^{p-2}|\nabla|\nabla f||^2.
\end{align*}
Choosing $\|\Ric_-^K\|^*_q$ small so that $\left((p-2)-C_s\|\Ric_-^K\|^*_q\frac{p^2}{4}\right)\geq 0$. Then we can throw the last term away and get
\begin{equation}
0\geq  - \frac{n-1+\sqrt{n}(p-2)}{n+ \sqrt{n}(p-2)}  \fint_M \Delta_pf \Delta f + ((n-1)K-2\|\Ric_-^K\|^*_q)\fint_M|\nabla f|^p.
\end{equation}

Let $f$ be the first eigenfunction for $\Delta_p$, that is, $\Delta_p f = -\mu|f|^{p-2}f$.  Then
\begin{align*}
\fint_M \Delta_pf\Delta f &= -\mu\fint_M |f|^{p-2}f\Delta f \\
&=\mu\fint_M \langle\nabla(|f|^{p-2}f),\nabla f\rangle \\
&=\mu(p-1)\fint_M |f|^{p-2}|\nabla f|^2 \\
&\leq (p-1)\mu\left(\fint_M |f|^p\right)^{1-\frac{2}{p}}\left(\fint_M|\nabla f|^p\right)^{\frac{2}{p}} \\
&= (p-1)(\mu)^{\frac{2}{p}}\fint_M|\nabla f|^p,
\end{align*}
where we use the fact that $f$ is the first eigenfunction, so we have
\begin{equation*}
\fint_M |f|^p = \frac{1}{\mu}\fint_M|\nabla f|^p.
\end{equation*}
This gives
\begin{equation*}
(\mu)^{\frac{2}{p}}  \left[ (p-1)  \frac{n-1+\sqrt{n}(p-2)}{n+ \sqrt{n}(p-2)}  \right] \ge (n-1)K -  2 \|\Ric_-^K\|^*_q,  
\end{equation*}
which is (\ref{N-eigen}). 
\end{proof}

\section{Proof of Theorem \ref{Faber-Krahn}}
To prove the Faber-Krahn-type estimate, we will need a version of the Gromov-Levy isoperimetric inequality for integral curvature. The inequality follows from the following volume comparison for tubular neighborhood of hypersurface of Petersen-Sprouse.
\begin{proposition}[\cite{petersen-sprouse}, Lemma 4.1 ]\label{petersen-sprouse}
Suppose that $H\subset M$ is a hypersurface with constant mean curvature $\eta \geq 0$, and that $H$ divides $M$ into two domains $\Omega_{\pm}$, where $\Omega_+$ is the domain in which mean curvature is positive.  Furthermore, let $d_{\pm} > 0$ such that $d_+ + d_- \leq \diam(M) \leq D$ and $\Omega_{\pm} \subset B(H,d_{\pm})$.  Let $\bar{H} = S(x_0,r_0) \subset \mathbb M^n_K$, a sphere of constant positive mean curvature $\eta$, and let $\bar{\Omega}_+ = B(x_0,D) - B(x_0,r_0)$, $\bar{\Omega}_- = B(x_0,r_0)$.  Finally assume that $d_+ \leq D- r_0$ and $d_- \leq r_0$. Then for any $\alpha > 1$, there is an $\eps(n,p,\alpha, K) > 0$ such that if $\|\Ric_-^K\|_q^* \leq \eps$, then
\begin{equation*}
\vol(\Omega_{\pm}) \leq \alpha \frac{\area(H)}{\area(\bar{H})}\vol(\bar{\Omega}_{\pm}).
\end{equation*}
\end{proposition}

Using this, the Gromov-Levy isoperimetric inequality for the integral curvature case can be shown by following the original proof given in \cite{gromov} page 522 and keeping track of the error term coming from the integral curvature.
\begin{proposition}\label{gromov-levy}
Let $\Omega\subset M$ be a domain.  Then for any $\alpha >1$, there is an $\eps = \eps(n,p,\alpha,K) >0 $ such that if $\|\Ric_-^K\|^*_{q}\leq \eps$, then
\begin{equation*}
\area(\dd B_K(r_0)) \leq \alpha \area (\dd\Omega)\frac{\vol(B_K(r_0))}{\vol(\Omega)} ,
\end{equation*}
where $B_K(r_0)\subset \mathbb{M}_K^n$ is the ball of radius $r_0$ in constant curvature $K$ space.  When $\|\Ric_-^K\|_q^* = 0$, we can take $\alpha =1$.
\end{proposition}

%\begin{proof}
%Consider all hypersurfaces $\tilde{H}\subset M$ which divide $M$ into $\tilde{\Omega}_0$, $\tilde{\Omega}_1$, with $\vol(\tilde{\Omega}_0) = \vol(\Omega)$ and $\vol(\tilde{\Omega}_1) = \vol(M) - \vol(\Omega)$.  Among these hypersurfaces, there is one with minimal area. \color{red}(reference? Almgren, Gromov)\color{black} If this surface is nonsingular, then it has constant mean curvature.  Let $\eta >0$ be the mean curvature relative to the normal in the direction $\tilde{\Omega}_0$. By Proposition \ref{petersen-sprouse}, we have for any $\alpha >1$, an $\eps >0$ such that if $\|\Ric_-^K\|_q^* \leq \eps$,
%\begin{align*}
%\vol(\Omega) = \vol(\tilde{\Omega}_0) \leq \alpha \frac{\area(\tilde{H})}{\area(\bar{H})}\vol(\bar{\Omega}_0) \leq \alpha \frac{\area(H)}{\area(\bar{H})}\vol(\bar{\Omega}_0)
%\end{align*}
%\end{proof}
Now we prove Theorem~\ref{Faber-Krahn},  the Faber-Krahn inequality.
%\begin{theorem}
%Let $M^n$ be a compact Riemannian manifold and let $\Omega \subset M$ be a domain and $B_K \subset M_K^n$ be a geodesic ball in a constant curvature space such that
%\begin{equation*}
%\frac{\vol(\Omega)}{\vol(M)} = \frac{\vol(B_K)}{\vol(M^n_K)}.
%\end{equation*}
%Then for any $\alpha >1 $ and for any $p> 1$, there is an $\eps = \eps(n,p,\alpha,K) >0$ such that if $\|\Ric_-^K\|_q^* \leq \eps$, then 
%\begin{equation*}
%\alpha^p\lambda_{1,p}(\Omega) \geq \lambda_{1,p}(B_K).
%\end{equation*}
%\end{theorem}
\begin{proof}
Without loss of generality, we can suppose that our test functions are Morse functions to ensure that the level sets of $f$ are closed regular hypersurfaces for almost all values.
%that 
%\begin{equation*}
%\frac{\int_\Omega |\nabla f|^p}{\int_\Omega |f|^p} \geq \frac{1}{\alpha^p}\lambda_{1,p}(B_K)
%\end{equation*}
 Let $\Omega_t := \{x \in \Omega \ | \ f > t\}$ and consider the decreasing rearrangement of $f$ defined by 
\begin{equation*}
\bar f(s) = \inf\{ t\geq 0 \ | \ |\Omega_t|<s \}
\end{equation*}
It is a non-increasing function on $[0,|\Omega|]$.  We define the spherical rearrangement $\bar\Omega$ of $\Omega$ as the ball in $\mathbb M_K^n$ centered at some fixed point such that  $\beta|\bar\Omega| = |\Omega|$, where $\beta := \frac{\vol(M)}{\vol (\mathbb M_K^n)}$. By abuse of notation, we define the spherical decreasing rearrangement $\bar f:\bar \Omega \to \mathbb{R}$ to be 
\begin{equation*}
\bar f(x) = \bar f(C_n|x|^n)
\end{equation*}
for $x \in \bar\Omega$, where $|x|$ is the distance from the center of $\bar\Omega$ and $C_n$ is the volume $S_K^n$. 
Note that
\begin{equation}\label{measureequality}
\vol(\{ f > t\}) = \vol (\{\bar f> t\}).
\end{equation}
 Now by construction, we have
\begin{equation*}
\int_{\Omega}f^p = \int_0^{|\Omega|}(\bar f(s))^pds = \beta\int_{\bar\Omega}(\bar f)^p.
\end{equation*}
Next we want to compare the $L^p$ norm of $\nabla f$ and $\nabla \bar f$.  Now $\dd\Omega_t = \{ x \in \Omega \ | \ f= t\}$ and since $\bar f$ is a radial function, we have
\begin{equation*}
|\nabla \bar f| = \left|\frac{\dd \bar f}{\dd r}\right| 
\end{equation*}
which is a constant on $\dd\bar \Omega_t$. By H\"older's inequality, we have
\begin{align*}
\vol(\{f=t\}) &= \int_{\{f=t\}} 1\\
&= \int_{\{f=t\}} \frac{|\nabla f|^{\frac{p-1}{p}}}{|\nabla f|^{\frac{p-1}{p}}} \\
&\leq \left(\int_{\{f=t\}} \frac{1}{|\nabla f|}  \right)^{\frac{p-1}{p}}\left( \int_{\{f=t\}} |\nabla f|^{p-1} \right)^{\frac{1}{p}}.
\end{align*}
By Proposition \ref{gromov-levy}
\begin{equation*}
\alpha \vol(\{f = t\}) \geq  \vol(\{\bar f = t\})
\end{equation*}
for some $\alpha > 1$. 
We have
\begin{align*}
\vol(\{\bar f=t\}) &= \int_{\{\bar f=t\}}1 \\
&= \left(\int_{\{\bar f=t\}} \frac{1}{|\nabla \bar f|}  \right)^{\frac{p-1}{p}}\left( \int_{\{\bar f=t\}} |\nabla \bar f|^{p-1} \right)^{\frac{1}{p}}.
\end{align*}
By the co-area formula, we have
\begin{align*}
\frac{\dd}{\dd t}\vol(\{f > t\}) &= \frac{\dd}{\dd t} \int_{\{f > t\}} 1 \\
&= \frac{\dd}{\dd t} \int_{t}^\infty \left(\int_{\{f=c\}}\frac{1}{|\nabla f|}\right)dc \\
&= -\int_{\{f=t\}}\frac{1}{|\nabla f|}.
\end{align*}
and similarly for $\bar f$ with \eqref{measureequality} so that
\begin{equation*}
\int_{\{f=t\}}\frac{1}{|\nabla f|} =  \int_{\{\bar f =t\}}\frac{1}{|\nabla \bar f|}.
\end{equation*}
Combining and applying the co-area formula once more to integrate over $\Omega$, we obtain
\begin{equation*}
\alpha^{p}\int_\Omega|\nabla f|^p \geq \int_{\bar \Omega} |\nabla \bar f|^p
\end{equation*}
and by the Rayleigh quotient, we have
\begin{align*}
\frac{\int_\Omega |\nabla f|^p}{\int_\Omega |f|^p} \geq \frac{1}{\alpha^p}\frac{\int_{\bar\Omega}|\nabla \bar f|^p}{\int_{\bar \Omega}|\bar f|^p} \geq \frac{1}{\alpha^p}\lambda_{1,p}(\bar \Omega).
\end{align*}
\end{proof}

To get Theorem \ref{optimal-est} from Theorem \ref{Faber-Krahn}, one follows the argument given in \cite{matei}.  One first shows the relation between the first non-trivial Neumann eigenvalue and the first Dirichlet eigenvalue of its nodal domain.  Namely, let $f$ be a first nontrivial Neumann eigenfunction of $\Delta_p$ on $M$ with $p>1$, let $A_+ = f^{-1}(\mathbb{R}_+)$ and $A_- = f^{-1}(\mathbb{R}_-)$ be the nodal domains of $f$.  Then
\begin{equation*}
\mu_{1,p}(M) = \lambda_{1,p}(A_+) = \lambda_{1,p}(A_-).
\end{equation*}
Using the fact that the nodal domains of $\Delta_p$ for the first nontrival Neumann eigenfunction on spheres $M_K^n$ are hemispheres $S^n_{K,\pm}$, in particular we have
\begin{equation*}
\mu_{1,p}(M^n_K) = \lambda_{1,p}(S^n_{K,\pm}).
\end{equation*}
Applying the Faber-Krahn-type estimate (Theorem \ref{Faber-Krahn}) to the nodal domain, we get Theorem \ref{optimal-est}.

\end{document}